\documentclass[12pt]{amsart}

\usepackage{amsfonts,amsmath,amssymb,amsthm}
\usepackage{enumerate}
\usepackage[right=3cm,left=3cm,top=3cm,bottom=3cm]{geometry}
\usepackage{mathrsfs}  
\usepackage{url}
\usepackage{xcolor}

\newtheorem{theorem}{Theorem}
\newtheorem{lemma}[theorem]{Lemma}
\newtheorem{proposition}[theorem]{Proposition}
\newtheorem{corollary}[theorem]{Corollary}

\theoremstyle{definition}
\newtheorem{definition}[theorem]{Definition}
\newtheorem{example}[theorem]{Example}

\newtheorem{notation}[theorem]{Notation}
\newtheorem{algorithm}[theorem]{Algorithm}

\title{Arithmetic varieties of numerical semigroups}

\author{Manuel B. Branco}
\address{Departamento de Matem\'aticas, Universidade de \'Evora, 7000-671 \'Evora, Portugal}
\email{mbb@uevora.pt}
\author{Ignacio Ojeda*}
\address{Departamento de Matem\'aticas, Universidad de Extremadura, 06071 Badajoz, Spain}
\email[*corresponding author]{ojedamc@unex.es}
\author{Jos\'e Carlos Rosales}
\address{Departamento de \'Algebra. Universidad de Granada, 18010 Granada, Spain}
\email{jrosales@ugr.es}
\subjclass[2020]{Primary: 20M14, 20M07 Secondary: 05C05}
\keywords{Numerical semigroup; varieties; rooted tree; Frobenius number; multiplicity; depth.}

\begin{document}

\begin{abstract}
In this paper we present the notion of arithmetic variety for numerical semigroups. We study various aspects related to these varieties such as the smallest arithmetic that contains a set of numerical semigroups and we exhibit the rooted tree associated with an arithmetic variety. This tree is not locally finite; however, if the Frobenius number is fixed, the tree has finitely many nodes and algorithms can be developed. All algorithms provided in this article include their (non-debugged) implementation in GAP.
\end{abstract}

\maketitle

\section{Introduction}

Let $\mathbb{Z}$ be the set of integer numbers
and let $\mathbb{N}$ be the set of non-negative integer numbers. A \emph{submonoid} of $(\mathbb{N},+)$ is a subset of $\mathbb{N}$ containing $0$ that is closed under addition. A \emph{numerical semigroup} is a submonoid $S$ of $(\mathbb{N},+)$ such that $\#(\mathbb{N} \setminus S) < \infty$, that is, $\mathbb{N} \setminus S$ has finite cardinality.

If $S$ is a numerical semigroup, then $\operatorname{m}(S) = \min(S \setminus \{0\}), \operatorname{F}(S) = \max(\mathbb{Z} \setminus S)$ and $\operatorname{g}(S) = \#(\mathbb{N} \setminus S)$ are relevant invariants  of $S$ called \emph{multiplicity}, \emph{Frobenius number} and \emph{genus of} $S$, respectively.

If $A$ is a non-empty subset of $\mathbb{N}$, then we write $\langle A \rangle$ for the submonoid of $(\mathbb{N},+)$ generated by $A$, that is, \[\langle A \rangle = \left\{u_1 a_1 + \cdots + u_n a_n \mid n \in \mathbb{N} \setminus \{0\}, \{a_1, \ldots, a_n\} \subseteq A\ \text{and}\ \{u_1, \ldots, u_n\} \subset \mathbb{N} \right\}.\] In \cite[Lemma 2.1]{libro} it is shown that $\langle A \rangle$ is a numerical semigroup if and only if $\gcd(A) = 1$.

If $M$ is a submonoid of $(\mathbb{N},+)$ and $M = \langle A \rangle$ for some non-empty subset $A$ of $\mathbb{N}$, then we say that $A$ is a system of generators of $M$. Moreover, if $M \neq \langle B \rangle$ for every $B \subsetneq A$, then we say that $A$ is a minimal system of generators of $M$. In \cite[Corollary 2.8]{libro} it is shown that every submonoid of $(\mathbb{N},+)$ has a unique minimal system of generators which, moreover, is finite. We write $\operatorname{msg}(M)$ for the minimal system of generators of $M$. The cardinality of $\operatorname{msg}(M)$ is the \emph{embedding dimension} of $M$ and is the denoted by $\operatorname{e}(M)$.

The \emph{Frobenius problem} for numerical semigroups (see \cite{alfonsin}) is to find formulas for the Frobenius number and the genus of a numerical semigroup in terms of its minimal system of generators. Nowadays this problem is widely open for numerical semigroups of embedding dimension greater than or equal to three.

Let $S$ and $T$ be numerical semigroups. Following the notation introduced in \cite{ojeda}, we say that $T$ is an \emph{arithmetic extension} of $S$ if there exist positive integers $d_1, \ldots, d_n$ such that \[T = \left\{ x \in \mathbb{N} \mid \{d_1 x, d_2 x, \ldots, d_n x\} \subset S\right\}.\] Notice that, in this case, we have that $S \subseteq T$.

\begin{definition}
An \emph{arithmetic variety} is a non-empty family $\mathscr{A}$ of numerical semigroups such that 
\begin{enumerate}[(a)]
    \item if $\{S,T\} \subseteq \mathscr{A}$, then $S \cap T \in \mathscr{A}$;
    \item if $S \in \mathscr{A}$ and $T$ is an arithmetic extension of $S$, then $T \in \mathscr{A}$.
\end{enumerate}
In this case, we say that $\mathscr{A}$ is a \emph{finite arithmetic variety} when $\mathscr{A}$ has finite cardinality.
\end{definition}

Notice that \[\mathscr{L}:=\{S \subseteq \mathbb{N} \mid S\ \text{is a numerical semigroup} \}\] is an arithmetic variety and $\mathcal{F} \subseteq \mathscr{L}$ for every family $\mathcal{F}$ of numerical semigroups. 

In the second section we prove that the intersection of arithmetic varieties is an arithmetic variety (Proposition \ref{Prop2}). Moreover, we emphasize that the intersection of all arithmetic varieties containing a given family $\mathcal{F}$ of numerical semigroups is an arithmetic variety, too (Proposition \ref{Prop3}). This arithmetic variety is denoted by $\mathscr{A}(F)$. We prove that $\mathscr{A}(\mathcal{F})$ is finite if and only if $\mathcal{F}$ has finite cardinality (Corollary \ref{Cor10}). Also, in this case, we give an algorithm  (Algorithm \ref{alg11}) to calculate all the elements of $\mathscr{A}(\mathcal{F})$.

In the third section we introduce the notion of $\mathscr{A}-$monoids and the minimal $\mathscr{A}-$sys\-tem of generators of a $\mathscr{A}-$monoid, where $\mathscr{A}$ is an arithmetic variety. Also, given $e \in \mathbb{N} \setminus \{0\}$, we write $\operatorname{ED}(e)$ for the set of numeric semigroups of the embedding dimension $e$. The results of the third section, combined with those of \cite{sistemas}, allow us to determine whether a numeric semigroup belongs to $\mathscr{A}(\operatorname{ED}(2))$. We propose the generalization to $\mathscr{A}(\operatorname{ED}(e))$ for $e \geq 3$ as an open problem.

If $S$ is a numerical semigroup then $\frac{S}2 := \left\{x \in \mathbb{N} \mid 2\, x \in S \right\}$ is a numerical semigroup (see \cite[Proposition 5.1]{libro}); in particular, it is an arithmetic extension of $S$. We write $\mathcal{D}_2(S)$ for set of numerical semigroups $T$ such that $S = \frac{T}{2}$. By \cite[Corollary 3]{R-GS2008}, this set is infinite and contains infinitely many symmetric numerical semigroups (see also \cite[Theorem 5]{swanson}). In the fourth section, we show that the elements in an arithmetic variety $\mathscr{A}$ can be arranged in the form of a tree $\mathcal{G}_\mathscr{A}$ with root $\mathbb{N}$ and such that the set of all the children of $S$ in the tree $\mathcal{G}_\mathscr{A}$ is equal to $\mathcal{D}_2(S) \cap \mathscr{A}$ (Theorem \ref{Th22}). Furthermore, we outline the description of $\mathcal{D}_2(S)$  given in \cite{dobles} (Theorem \ref{Th24}), because of its usefulness in the following sections.

If $\mathscr{A}$ is an arithmetic variety and $F$ is a positive integer, we define \[\mathscr{A}_{F} := \{S \in \mathscr{A} \mid \operatorname{F}(S) \leq F\}.\] In the fifth section, we will see that $\mathscr{A}_{F}$ is a finite arithmetic variety (Proposition \ref{Prop30}); moreover, we give an algorithm (Algorithm \ref{Alg37}) to compute $\{T \in \mathcal{D}_2(S) \mid \operatorname{F}(T) \leq F\}$.

The \emph{depth} of a numerical semigroup $S$, denoted $\operatorname{depth}(S)$, is equal to $\left\lceil \frac{\operatorname{F}(S)+1}{\operatorname{m}(S)} \right\rceil$, where $\left\lceil q \right\rceil$ is the ceiling function (the smallest integer greater than $q$). The depth of a numerical semigroup was recently introduced in \cite{gapsets} where evidences are given that support the Bras-Amor\'os conjecture (\cite{bras}). Also, it is proved that numerical semigroups of depth less than or equal to three are Wilf (see \cite{delgado} for further details on Wilf's conjecture). Moreover, following \cite[Corollary 21]{questiones} and the terminology introduced therein, one can see that the complexity of a numerical semigroup is equal to its depth.

If $q \in \mathbb{N}$, then we write $\mathscr{C}_q$ for the set of numerical semigroups with depth less than or equal to $q$ . In the sixth section, we prove that $\mathscr{C}_q$ is an arithmetic variety (Theorem \ref{ThDepth}). Furthermore, taking advantage of the results of the third Section, we formulate an algorithm (Algorithm \ref{Alg43}) to compute the subset of $\mathscr{C}_q$ consisting of numerical semigroups with Frobenius number $F$.

\section{The smallest arithmetic variety containing a family of numerical semigroups}

Let $S$ be a numerical semigroup and $d \in \mathbb{N} \setminus \{0\}$. As mentioned in the introduction, we write $\frac{S}{d}$ for the set $\left\{x \in \mathbb{N} \mid d\, x \in S\right\}.$ In \cite[Proposition 5.1]{libro} it is shown that $\frac{S}{d}$ is a numerical semigroup. This semigroup is called the \emph{quotient of $S$ by $d$}.

Notice that $\frac{S}{d} = \mathbb{N}$ if and only if $d \in S$. Also, by definition, $T$ is an arithmetic extension of $S$ if and only if there exists $\{d_1, \ldots, d_n\} \subset \mathbb{N} \setminus \{0\}$ such that $T = \frac{S}{d_1} \cap \cdots \cap \frac{S}{d_n}$. With these remarks, the proof of the following result is straightforward.

\begin{proposition}\label{Prop1}
Let $\mathscr{A}$ be a non-empty family of numerical semigroups. Then $\mathscr{A}$ is an arithmetic variety if and only if the following hold:
\begin{enumerate}[(a)]
    \item if $\{S,T\} \subset \mathscr{A}$, then $S \cap T \in \mathscr{A}$;
    \item if $S \in \mathscr{A}$ and $d \in \mathbb{N} \setminus \{0\}$, then $\frac{S}{d} \in \mathscr{A}$.
\end{enumerate}
\end{proposition}

Since all arithmetic varieties contains $\{\mathbb{N}\}$, we have that the intersection of arithmetic varieties is a non-empty set of numerical semigroups.

\begin{proposition}\label{Prop2}
The intersection of arithmetic varieties is an arithmetic variety.
\end{proposition}

\begin{proof}
Let $\{\mathscr{A}_i\}_{i \in I}$ be an arbitrary family of arithmetic varieties. From the previous observation, it follows that $\mathbb{N} \in \cap_{i \in I} \mathscr{A}_i$. Thus, the set $\cap_{i \in I} \mathscr{A}_i$ is non-empty; let us see that it is an arithmetic variety. On the one hand, if $\{S,T\} \subseteq \cap_{i \in I} \mathscr{A}_i$, then $\{S,T\} \subseteq \mathscr{A}_i$ for every $i \in I$, therefore $S \cap T \in \mathscr{A}_i$ for every $i \in I$ and, consequently, $S \cap T \in \cap_{i \in I} \mathscr{A}_i$. On the other hand, if $S \in \cap_{i \in I} \mathscr{A}_i$ and $d \in \mathbb{N} \setminus \{0\}$, then, by Proposition \ref{Prop1}, we have that $\frac{S}{d} \in \mathscr{A}_i$ for every $i \in I$; hence $\frac{S}{d} \in \cap_{i \in I} \mathscr{A}_i$. So, applying  Proposition \ref{Prop1} again, we are done.
\end{proof}

Recall that, if $\mathcal{F}$ is a family of numerical semigroups, then we write $\mathscr{A}(\mathcal{F})$ to denote the intersection of all arithmetic varieties containing $\mathcal{F}$. Therefore, by Proposition \ref{Prop2}, we have the following.

\begin{proposition}\label{Prop3}
If $\mathcal{F}$ is a family of numerical semigroups, then $\mathscr{A}(\mathcal{F})$ is the smallest arithmetic variety containing $\mathcal{F}$.
\end{proposition}

The following result is technical and its proof is carried out by direct verification.

\begin{lemma}\label{Lema4}
If $S, T$ are numerical semigroups and $a,b$ are positive integers, then 
\[
\frac{\frac{S}{a}}{b} = \frac{S}{ab}\quad \mbox{and}\quad \frac{S \cap T}{a} = \frac{S}{a} \cap \frac{T}{a}.
\]
\end{lemma}

\begin{lemma}\label{Lema5}
If $S$ is a numerical semigroup, then 
\[
\mathscr{A} = \left\{ \bigcap_{i=1}^n \frac{S}{d_i} \mid n \in \mathbb{N}\setminus \{0\}\ \text{and}\ \{d_1, \ldots, d_n\} \subset \mathbb{N} \setminus S \right\} \cup \{\mathbb{N}\}
\]
is an arithmetic variety containing $\{S\}$. 
\end{lemma}

\begin{proof}
Clearly, the intersection of any two elements in $\mathscr{A}$ belongs to $\mathscr{A}$ and, by Lemma \ref{Lema4}, we have that $\frac{S}{d} \in \mathscr{A}$ for every $S \in \mathscr{A}$ and $d \in \mathbb{N} \setminus \{0\}$. Thus, by Proposition \ref{Prop1}, $\mathscr{A}$ is an arithmetic variety. Moreover, since $S = \frac{S}{1}$, we conclude that $\{S\} \subseteq \mathscr{A}$.
\end{proof}

\begin{proposition}\label{Prop6}
If $S$ is a numerical semigroup, then 
\begin{equation}\label{ecu1}
\mathscr{A}(\{S\}) = \left\{ \bigcap_{i=1}^n \frac{S}{d_i} \mid n \in \mathbb{N}\setminus \{0\}\ \text{and}\ \{d_1, \ldots, d_n\} \subset \mathbb{N} \setminus S \right\} \cup \{\mathbb{N}\}.
\end{equation}
\end{proposition}

\begin{proof}
Let $\mathscr{A}$ be the right hand side of \eqref{ecu1} and let $\mathscr{A}'$ be an arithmetic variety containing $\{S\}.$ From Proposition \ref{Prop1} it follows that $\mathscr{A} \subseteq \mathscr{A}'$. Therefore, by Lemma \ref{Lema5}, we have that $\mathscr{A}$ is the smallest arithmetic variety containing $\{S\}$. Now, by Proposition \ref{Prop3}, we conclude that $\mathscr{A} = \mathscr{A}(\{S\})$.
\end{proof}

The following result is an immediate consequence of Propositions \ref{Prop1} and \ref{Prop6} (see also \cite[Proposition 1]{ojeda}).

\begin{corollary}\label{Lema15}
If $S$ is a numerical semigroup, then \[\mathscr{A}(\{S\}) = \{T \in \mathscr{L} \mid T\ \text{is an arithmetic extension of}\ S\}.\]
\end{corollary}

\begin{proposition}\label{Prop8}
If $S$ is a numerical semigroup, then $\mathscr{A}(\{S\})$ is a finite arithmetic variety.
\end{proposition}

\begin{proof}
By Proposition \ref{Prop3}, we have that $\mathscr{A}(\{S\})$ is an arithmetic variety and, by Proposition \ref{Prop6}, we have that $\mathscr{A}(\{S\}) \subseteq \mathcal{F} := \{T \in \mathscr{L} \mid S \subseteq T\}$. Now, since $\mathbb{N} \setminus S$ has finite cardinality, because $S$ is a numerical semigroup, we conclude that $\mathcal{F}$ is a finite set and our claim follows.
\end{proof}

Notice that, by Proposition \ref{Prop6}, we can use Algorithm 23 in \cite{ojeda} to compute $\mathscr{A}(\{S\})$ from $S$.

\begin{theorem}\label{th9}
If $\mathcal{F}$ is a non-empty family of numerical semigroups, then 
\begin{equation}\label{ecu2}
\mathscr{A}(\mathcal{F}) = \left\{ \bigcap_{i=1}^n T_i \mid n \in \mathbb{N} \setminus \{0\}\ \text{and}\ T_i \in \mathscr{A}(\{S_i\})\ \text{for some}\ S_i \in \mathcal{F}, i = 1, \ldots, n \right\}.
\end{equation}
\end{theorem}

\begin{proof}
Let $\mathscr{A}$ be the right hand side of \eqref{ecu2}. Clearly, we have that $\mathcal{F} \subseteq \mathscr{A} \subseteq \mathscr{A}'$, for every arithmetic variety $\mathscr{A}'$ containing $\mathcal{F}$. Thus, by Proposition \ref{Prop3}, to see that $\mathscr{A}(\mathcal{F}) = \mathscr{A}$ it suffices to prove that $\mathscr{A}$ is an arithmetic variety. Of course, if $\{S,T\} \subseteq \mathscr{A}$, then $S \cap T \in \mathscr{A}$ and, by Lemma \ref{Lema4}, it is easy to check that $\frac{S}{d} \in \mathscr{A}$, for every 
$S \in \mathscr{A}$ and $d \in \mathbb{N} \setminus \{0\}$. Therefore, by Proposition \ref{Prop1}, we conclude that $\mathscr{A}$ is an arithmetic variety. 
\end{proof}

\begin{corollary}\label{Cor10}
Let $\mathcal{F}$ be a family of numerical semigroups. Then $\mathscr{A}(\mathcal{F})$ is a finite arithmetic variety if and only if $\mathcal{F}$ has finite cardinality.
\end{corollary}

Now, by combining \cite[Algorithm 23]{ojeda} and Theorem \ref{th9}, we obtain an algorithm to compute $\mathscr{A}(\mathcal{F})$, provided that $\mathcal{F}$ is a finite family of numerical semigroups.

\begin{algorithm}\label{alg11} Computation of $\mathscr{A}(\mathcal{F})$.

\textsc{Input:} A finite set $\mathcal{F} = \{S_1, \ldots, S_n\}$ of numerical semigroups. 

\textsc{Output:} $\mathscr{A}(\mathcal{F})$.

\begin{enumerate}
    \item Set $\mathscr{A}(\mathcal{F}) = \{\mathbb{N}\}$.
    \item For each $i \in \{1, \ldots, n\}$, set $\mathscr{A}_i = \mathscr{A}(\{S_i\})$.
    \item For each $(T_1, \ldots, T_n) \in \mathscr{A}_1 \times \cdots \times \mathscr{A}_n$, do \[\mathscr{A}(\mathcal{F}) = \mathscr{A}(\mathcal{F}) \cup \{T_1 \cap \cdots \cap T_n\}.\]
    \item Return $\mathscr{A}(\mathcal{F})$.
\end{enumerate}
\end{algorithm}

\begin{example}\label{ex12}
Let $\mathcal{F} = \{\langle 2,5 \rangle, \langle 3,5,7 \rangle\}$. By \cite[Algorithm 23]{ojeda}, we have that 
\[\mathscr{A}(\{\langle 2,5 \rangle\} = \{\mathbb{N}, \langle 2,3 \rangle, \langle 2,5 \rangle\}\]
and that
\[\mathscr{A}(\{\langle 3,5,7 \rangle\} = \{\mathbb{N}, \langle 2,3 \rangle, \langle 
3,4,5 \rangle, \langle 
3,5,7 \rangle\}.\] Therefore, by Algorithm \ref{alg11}, we conclude that \[\mathscr{A}(\mathcal{F}) = \{ \mathbb{N}, \langle 2,3 \rangle, 
 \langle 2,5 \rangle,  \langle 3,4,5 \rangle,  \langle 3,5,7 \rangle,  \langle 4,5,6,7 \rangle,  \langle 5,6,7,8,9 \rangle \}.\]
\end{example}

The function \texttt{ArithmeticExtensions}, given by the second and third authors in \cite[pp. 3714--3715]{ojeda}, uses the package \texttt{NumericalSpgs} (\cite{numericalsgps}) of GAP (\cite{GAP}) to calculate $\mathscr{A}(\{S\})$ with $S$ being a numerical semigroup. Therefore, by Algorithm \ref{alg11}, we can compute $\mathscr{A}(\mathcal{F})$, with $\mathcal{F}$ being a family of numerical semigroup, by the following code:

{\small
\begin{verbatim}
  SmallestArithmeticVariety:=function(F)
    local AF,A,S;
    AF:=[NumericalSemigroup(1)];
    A:=[];
    for S in F do 
      Append(A,[ArithmeticExtensions(S)]);
    od;
    Append(AF,List(Cartesian(A),i->Intersection(i)));
    return Set(AF);
  end;  
\end{verbatim}
} 
\noindent For example, if $\mathcal{F} = \{\langle 2,5 \rangle, \langle 3,5,7 \rangle\}$ we write
{\small
\begin{verbatim}
  F:=[[2,5],[3,5,7]];
  F:=List(F,i->NumericalSemigroup(i));
  SmallestArithmeticVariety(F);
\end{verbatim}
}
\noindent provided that the package \texttt{NumericalSgps} and the function \texttt{ArithmeticExtensions} have already been loaded into GAP. 

\section{$\mathscr{A}-$system of generators}

Throughout this section, $\mathscr{A}$ denotes an arithmetic variety. By Proposition \ref{Prop1}, the intersection of finitely many elements in $\mathscr{A}$ is an element of $\mathscr{A}$. This does not occur at the intersection of infinitely many elements, as the following example evidences.

\begin{example}\label{ex13}
The set $\mathscr{A} = \left\{ \{0,n,n+1, \ldots \} \mid n \in \mathbb{N} \setminus \{0\} \right\}$ is an arithmetic variety; however, $\bigcap_{n \in \mathbb{N} \setminus \{0\}} \{0, n, n+1, \ldots \} = \{0\} \not\in \mathscr{A}$.    
\end{example}

Despite the previous example, the arbitrary intersection of elements in $\mathscr{A}$ is always a submonoid of $(\mathbb{N},+)$. This fact gives meaning to the following definition.

\begin{definition}
Given an arithmetic variety $\mathscr{A}$. An \emph{$\mathscr{A}-$monoid} is a submonoid of $(\mathbb{N},+)$ that can be written as an intersection of elements of $\mathscr{A}$.
\end{definition}

Thus, given $X \subseteq \mathbb{N}$, we have that the intersection of all elements in $\mathscr{A}$ containing $X$ is the \emph{smallest $\mathscr{A}-$monoid containing $X$} that we denote by $\mathscr{A}[X]$. Now, if $M$ is an $\mathscr{A}-$monoid such that $M = \mathscr{A}[X]$, then we say that $X$ is a \emph{$\mathscr{A}-$system of generators of $M$}. Moreover, if $M \neq \mathcal{A}[Y]$, for every $Y \subsetneq X$, then we say that $X$ is a \emph{minimal $\mathscr{A}-$system of generators of $M$}.

Let us see that there are $\mathscr{A}-$monoids having non-unique minimal $\mathscr{A}-$systems of generators, for a given arithmetic variety $\mathscr{A}$; but let us first recall the notion of fundamental gap and a result from \cite{ojeda}.

\begin{definition}
Let $S$ be a numerical semigroup. An element $x \in \mathbb{N} \setminus \{S\}$ is a \emph{fundamental gap} of $S$ if $\{k\, x \mid k \in \mathbb{N} \setminus \{1\}\} \subseteq S$. We write $\operatorname{FG}(S)$ for the set of fundamental gaps of $S$. 
\end{definition}

By Corollary \ref{Lema15}, the following result is nothing more than a reformulation of \cite[Proposition 6]{ojeda}, we include it here for complete exposition and ease of reading.

\begin{proposition}\label{Prop16}
If $S \neq \mathbb{N}$ is a numerical semigroup, then the following hold:
\begin{enumerate}[(a)]
    \item $\max_{\subseteq}(\mathscr{A}(\{S\})) = \mathbb{N}$,
    \item $\min_{\subseteq}(\mathscr{A}(\{S\})) = S$,
    \item $\max_{\subseteq}(\mathscr{A}(\{S\}) \setminus \{\mathbb{N}\}) = \langle 2,3 \rangle$,
    \item $\min_{\subseteq}(\mathscr{A}(\{S\}) \setminus \{S\}) = S \cup \operatorname{FG}(S).$
\end{enumerate}
\end{proposition}

\begin{corollary}\label{Cor_FG}
Let $S \neq \mathbb{N}$ be a numerical semigroup. Then $\mathscr{A}(\{S\})[\{x\}] = 
S \cup \operatorname{FG}(S)$, for every $x \in \operatorname{FG}(S)$. 
\end{corollary}

\begin{proof}
If $x \in \operatorname{FG}(S)$, then $x \not\in S$. So, by Proposition \ref{Prop16}, the smallest element of $\mathscr{A}(\{S\})$ that contains $\{x\}$ is $S \cup \operatorname{FG}(S)$ .
\end{proof}

\begin{example}\label{ex17}
Let $S = \langle 5,7,9 \rangle$. By direct computation, one can check that $\operatorname{FG}(S) = \{6,8,11,13\}$. Therefore, by Corollary \ref{Cor_FG}, we have that $\{6\}, \{8\}, \{11\}$ and $\{13\}$ are minimal $\mathscr{A}(\{S\})-$systems of generators of $S \cup \operatorname{FG}(S) = \langle 5,6,7,8,9 \rangle$.
\end{example}

Despite the previous example, there are arithmetic varieties, $\mathcal{A}$, in which all $\mathcal{A}-$mo\-noids have unique minimal $\mathcal{A}-$system of generators. To show one of them, we first recall several notions and results on proportionally modular numerical semigroups and their generalizations.

Let $a, b$ and $c$ be positive integers. If $a x\!\! \mod b$ denotes the remainder of the Euclidean division of $a x$ by $b$, the set \[\{x \in \mathbb{N} \mid a x\!\!\! \mod b \leq c x\}\] is a numerical semigroup called \emph{proportionally modular numerical semigroup}  (see \cite{proportional, modulares} for more details).

Recall that $\operatorname{ED}(e) = \{S \in \mathscr{L} \mid \operatorname{e}(S) = e\}$ is the set of numerical semigroups of embedding dimension $e$. The following results follow from \cite[Proposition 41]{sistemas} and \cite[Theorem 12]{sistemas}, respectively.

\begin{proposition}\label{Prop18}
The arithmetic variety $\mathscr{A}(\operatorname{ED}(2))$ is equal to the set of intersections of finitely many proportionally modular numerical semigroups.
\end{proposition}

\begin{corollary}
Every $\mathscr{A}(\operatorname{ED}(2))-$monoid has a unique minimal $\mathscr{A}(\operatorname{ED}(2))-$system of generators. 
\end{corollary}

We finish this section by recalling and proposing some open problems.

\subsection*{Some open problems} 
In \cite[Theorem 5]{full}, it is proved that a numerical semigroup is proportionally modular if and only if it is the quotient of a numerical semigroup of embedding dimension $2$ by a positive integer. So, \cite[Theorem 31]{theset} provides an algorithm to decide whether a numerical semigroup belongs to $\left\{\frac{S}d \mid S \in \operatorname{ED}(2)\ \text{and}\ d \in \mathbb{N} \setminus \{0\}\right\}$.

In \cite{problem}, the problem of finding a numerical semigroup that cannot be written as the quotient of a element of $\operatorname{ED}(3)$ by a positive integer is proposed. In \cite{canadian}, it is proved its existence; however no example is given. Recently, in \cite{australia}, some examples are exhibited. Have an algorithm to decide whether a numerical semigroup belongs to $\left\{\frac{S}d \mid S \in \operatorname{ED}(3)\ \text{and}\ d \in \mathbb{N} \setminus \{0\}\right\}$ is still an open problem.

By Proposition \ref{Prop18} and the results in \cite{sistemas}, one can deduce an algorithm to decide whether a numerical semigroup belongs to $\mathscr{A}(\operatorname{ED}(2))$. We propose as an open problem to formulate the corresponding algorithm for $\mathscr{A}(\operatorname{ED}(3))$ and, being optimistic, for $\mathscr{A}(\operatorname{ED}(e)),\ e \geq 4$. 

\section{The tree associated with an arithmetic variety}

If $\mathscr{A}$ is an arithmetic variety, then we define the directed graph $\mathcal{G}_\mathscr{A}$, whose vertex set is $\mathscr{A}$, having an edge from $T \in \mathscr{A}$ to $S \in \mathscr{A} \setminus \{\mathbb{N}\}$ if and only if $T = \frac{S}2$; equivalently, such that the set of children of $S \in \mathscr{A} \setminus \{\mathbb{N}\}$ is $\mathcal{D}_2(S) \cap \mathscr{A}$, where $\mathcal{D}_2(S) = \{T \in \mathscr{L} \mid S = \frac{T}2\}$. 

\begin{theorem}\label{Th22}
If $\mathcal{A}$ is an arithmetic variety, then $\mathcal{G}_\mathscr{A}$ is a directed rooted tree with root $\mathbb{N}$.
\end{theorem}

\begin{proof}
Recall that a directed rooted tree is a directed graph such that for each vertex there is a unique directed path from or towards a single vertex called root.

First, we notice that $\mathcal{G}_\mathscr{A}$ has no loops. Indeed, if $S \neq \mathbb{N}$, then $\operatorname{F}(S) \in \frac{S}2$ which implies $S \subsetneq \frac{S}2$. Now, given $S \in \mathscr{A} \setminus \{\mathbb{N}\}$, consider the sequence $\{S_n\}_{n \in \mathbb{N}}$ such that $S_0 = S$ and $S_{n+1} = \frac{S_n}2,$ for every $n \in \mathbb{N}$. Since, by Lemma \ref{Lema4}, $S_n = \frac{S}{2^n}$, we have that $S_n \in \mathcal{A}$, for every $n \in \mathbb{N}$, by Proposition \ref{Prop1}. Moreover, since $S_n \subsetneq S_{n+1}$, whenever $S_n \neq \mathbb{N}$ and $\mathbb{N} \setminus S$ has finite cardinality, we conclude that there exists $k \in \mathbb{N}$ such that $S_k = \mathbb{N}$ and $S_{k-1} \subsetneq S_k$. This proves the existence of a directed path in $\mathcal{G}_\mathscr{A}$ from $\mathbb{N}$ to $S \in \mathcal{A}$, the uniqueness follows by the own definition of $\mathcal{G}_\mathscr{A}$.
\end{proof}

In \cite{R-GS2008}, it is shown that $\mathcal{D}_2(S)$ is an infinite set for every $S \in \mathscr{L} \setminus \{\mathbb{N}\}$ which implies that $\mathcal{G}_\mathscr{A}$ is not locally finite. Therefore, it is not possible to give a general algorithm for the computation of the tree $\mathcal{G}_\mathscr{A}$ starting from the root nor from any other parent. 

Nevertheless, according to \cite[Theorem 7]{dobles}, one can describe what the elements of $\mathcal{D}_2(S)$ are like in terms of the so-called upper $m-$-sets of $S$.  Let us remember this construction, which will be useful in the next sections.

First of all, we notice that 
\[\mathcal{D}_2(\mathbb{N}) = \left\{ \langle 2 ,2n+1 \rangle \mid n \in \mathbb{N} \right\}.\] So, we only need to describe $\mathcal{D}_2(S)$ for $S \neq \mathbb{N}$. 

\begin{definition}\label{Def_ums}
Let $S \subsetneq \mathbb{N}$ be a numerical semigroup and let $m$ be an odd element of $S$. An \emph{upper $m-$set} of $S$ is a subset $H$ of $\mathbb{N} \setminus S$ such that
\begin{enumerate}[(C1)]
    \item $\{h + m \mid h \in H\} \subseteq S$;
    \item $\{h_1 + h_2 + m \mid h_1, h_2 \in H\} \subseteq S$;
    \item $h \in H \Longrightarrow \{x \in \mathbb{N} \setminus S \mid x-h \in S \} \subseteq H$. 
\end{enumerate} 
\end{definition}

Given a triplet $(S,m,H)$, where $S \neq \mathbb{N}$ is a numerical semigroup, $m$ is an odd element of $S$ and $H \subseteq \mathbb{N} \setminus S$, the following the GAP function decides whether $H$ is an upper $m-$set of $S$.

{\small
\begin{verbatim}
  IsUppermSetOfNumericalSemigroup:=function(S,m,H)
    local C1,C2,C3,h;
    C1:=Intersection(H+m,Gaps(S));
    if IsEmpty(C1) = false then return(false); fi;
    C2:=Intersection(Set(Cartesian(H,H),i->Sum(i)+m),Gaps(S));
    if IsEmpty(C2) = false then return(false); fi;
    for h in H do
      C3:=Filtered(Gaps(S),i->BelongsToNumericalSemigroup(i-h,S));
      if not(Intersection(C3,H)=C3) then return(false); fi;
    od;
    return(true);
  end;
\end{verbatim}
}

\begin{example}
Using the script above one can check, as described below, that the set of upper $5-$sets of $S = \langle 4,5,11 \rangle$ are \[\{\{3,7\}, \{3,6,7\}, \{6\}, \{7\}, \{6,7\}\}.\] 
{\small
\begin{verbatim}
  LoadPackage("NumericalSgps");
  S:=NumericalSemigroup(4,5,11);
  pow:=Combinations(Gaps(S));
  pow:=Difference(pow,[[]]);
  Filtered(pow,H->IsUppermSetOfNumericalSemigroup(S,5,H));
\end{verbatim}
}
\end{example}

\begin{notation}\label{Not25}
If $H$ is a \emph{upper $m-$set} of a numerical semigroup $S$ and $m \in S$ is odd, we write $S(m,H)$ for \[\{2s \mid s \in S\} \cup \{2s + m \mid s \in S\} \cup \{2 h + m \mid h \in H\}.\] Thus, if $\operatorname{msg}(S) = \{a_1, \ldots, a_e\}$, then $S(m,H)$ is generated by $\{2 a_1, \ldots, 2 a_e\} \cup \{m\} \cup \{2 h + m \mid h \in H\}$.
\end{notation}

The following result is \cite[Theorem 7]{dobles}.

\begin{theorem}\label{Th24}
If  $S \subsetneq \mathbb{N}$ is a numerical semigroup, then 
\[\mathcal{D}_2(S) = \{ S(m,H) \mid m\ \text{is an odd element of}\ S\ \text{and}\ H\ \text{is an upper $m-$set of}\ S\}.\]
\end{theorem}

\begin{example}
If $S = \langle 2, 3 \rangle$, then the only upper $m-$set of $S$ is $H=\{1\}$ for every odd integer number $m$ greater than one. Therefore, by Theorem \ref{Th24}, 
\begin{align*}
\mathcal{D}_2(S) & = \left\{ \langle 4,6,4+2n+1,6+2n+1,2+ 2n+1 \rangle \mid n \in \mathbb{N} \setminus \{0\} \right\} \\  & = \left\{ \langle 4,6,2n+3,2n+5 \rangle \mid n \in \mathbb{N} \setminus \{0\} \right\}.
\end{align*}
\end{example}

We finish this section by recalling \cite[Corollary 8]{dobles} which states that $S(m_1,H_1) = S(m_2,H_2)$ if and only if $m_1 = m_2$ and $H_1 = H_2$, giving rise to a new proof of the infinite cardinality of $\mathcal{D}_2(S)$ (for more details see \cite[Corollary 14]{dobles}).

\section{The elements of an arithmetic variety with bounded Frobenius number}

If $\mathscr{A}$ is an arithmetic variety and $F$ is a positive integer, we define 
\[\mathscr{A}_F := \{S \in \mathscr{A} \mid \operatorname{F}(S) \leq F\}.\]

\begin{proposition}\label{Prop30}
If $\mathscr{A}$ is an arithmetic variety and $F$ is a positive integer, then $\mathscr{A}_F$ is a finite arithmetic variety.
\end{proposition}

\begin{proof}
Since $\operatorname{F}(\mathbb{N}) = -1,$ we have that $\mathbb{N} \in \mathscr{A}_F$; then, $\mathscr{A}_F \neq \varnothing$. Moreover, if $S \in \mathscr{A}_F$, then $\{0,F+1,F+2, \ldots \} \subseteq S$ and, consequently, $\mathscr{A}_F $ is a finite set. Now, let us use Proposition \ref{Prop1} to show that $\mathscr{A}_F$ is an arithmetic variety. On the one hand, if $\{S,T\} \subseteq \mathscr{A}_F$, then $S \cap T \in \mathscr{A}$ and, since $\operatorname{F}(S \cap T) = \max\{\operatorname{F}(S), \operatorname{F}(T)\}$, we conclude that $S \cap T \in \mathscr{A}_F$. On the other hand, if $S \in \mathscr{A}_F$ and $d \in \mathbb{N} \setminus \{0\}$, then $\frac{S}d \in \mathscr{A}_F$; thus, $\operatorname{F}\left(\frac{S}d\right) \leq \operatorname{F}(S) \leq F$ and, consequently, $\frac{S}d \in \mathscr{A}_F$.
\end{proof}

Notice that  $\mathcal{G}_{\mathscr{A}_F}$ is now finite, because it has finitely many nodes. In particular the set, $\mathcal{D}_2(S) \cap \mathscr{A}_F$, of children of $S \in \mathscr{A}_F$ in $\mathcal{G}_{\mathscr{A}_F}$ is now finite. Furthermore, we have the following.

\begin{proposition}\label{Prop31}
Let $\mathscr{A}$ be an arithmetic variety and let $F$ be a positive integer. If $S \in \mathscr{A}_F$, then $\mathcal{D}_2(S) \cap \mathscr{A}_F = \{T \in \mathcal{D}_2(S) \mid \operatorname{F}(T) \leq F\} \cap \mathscr{A}$.
\end{proposition}

\begin{proof}
If $T \in \mathcal{D}_2(S) \cap \mathscr{A}_F$, then $T \in \mathcal{D}_2(S), \operatorname{F}(T) \leq F$ and $T \in \mathscr{A}$. Conversely, if $T \in \mathcal{D}_2(S), \operatorname{F}(T) \leq F$ and $T \in \mathscr{A}$, then $T \in \mathcal{D}_2(S) \cap \mathscr{A}_F$.
\end{proof}

In view of the previous result, an algorithm is obtained for the computation of $\mathcal{G}_{\mathscr{A}_F}$ provided that we can compute $\{T \in \mathcal{D}_2(S ) \mid \operatorname{F}(T) \leq F\}$ for a given $S \in \mathscr{A}_F$. The rest of the section is devoted to this purpose.

The following result follows from \cite[Proposition 9]{dobles}.

\begin{proposition}\label{Prop32}
Let $S \neq \mathbb{N}$ be a numerical semigroup. If $m$ is an odd element of $S$ and $H$ is an upper $m-$set of $S$, then
\[
\operatorname{F}(S(m,H)) = \left\{\begin{array}{ll}
\max(2 \operatorname{F}(S), m-2) & \text{if}\ H = \mathbb{N} \setminus S; \\
\max(2 \operatorname{F}(S), 2 \max(\mathbb{N} \setminus S \cup H) +m)) &  \text{if}\ H \neq \mathbb{N} \setminus S
\end{array}\right.
\]
\end{proposition}

\begin{lemma}\label{Lema33}
Let $S \neq \mathbb{N}$ be a numerical semigroup and let $m$ be an odd integer greater that $\operatorname{F}(S)$. If $H \subseteq \mathbb{N} \setminus S$ satisfies condition (C3) in Definition \ref{Def_ums}, then $H$ is an upper $m-$set of $S$.
\end{lemma}

\begin{proof}
It suffices to observe that if $m \geq \operatorname{F}(S),$ then $\{h + m \mid h \in H\} \subseteq S$ and $\{h_1 + h_2 + m \mid h_1, h_2 \in H\} \subseteq S$.
\end{proof}

\begin{proposition}\label{Prop34}
Let $S \neq \mathbb{N}$ be a numerical semigroup and let $m$ be an odd element of $S$. Then $\mathbb{N} \setminus S$ is an upper $m-$set of $S$ if and only if $m > \operatorname{F}(S)$.
\end{proposition}

\begin{proof}
If $\mathbb{N} \setminus S$ is an upper $m-$set of $S$, then $\{g + m \mid g \in \mathbb{N} \setminus S\} \subseteq S$. Now, since $\{1, \ldots, \operatorname{m}(S)-1\} \subseteq \mathbb{N} \setminus S,$ then $\{m, m+1, \ldots, m + \operatorname{m(S)} -1, m + \operatorname{m(S)}, \ldots \} \subseteq S$ and, therefore $m > \operatorname{F}(S)$. Conversely, if  $m > \operatorname{F}(S)$, then $\mathbb{N} \setminus S$ satisfies condition (C3) in Definition \ref{Def_ums}. Thus, by Lemma \ref{Lema33}, we conclude that $\mathbb{N} \setminus S$ is an upper $m-$set of $S$.
\end{proof}

Combining Theorem \ref{Th24} and Proposition \ref{Prop32}, we obtain the following immediate result.

\begin{lemma}\label{Lema35}
Let $S$ be a numerical semigroup and let $F$ be a positive number. If $2 \operatorname{F}(S) > F$, then $\{ T \in \mathcal{D}_2(S) \mid \operatorname{F}(T) \leq F\} = \varnothing$.
\end{lemma}

\begin{theorem}\label{Th36}
Let $S$ be a numerical semigroup and let $F$ be a positive number. If $2 \operatorname{F}(S) \leq F$, then $\{T \in \mathcal{D}_2(S ) \mid \operatorname{F}(T) \leq F\}$ is equal to the union of
\[\left\{S(m,\mathbb{N}\setminus S) \mid m\ \text{is an odd element of}\ \{\operatorname{F}(S)+1, \ldots, F+2\}\right\}\]
and
\[\left\{ S(m,H) \left\vert \begin{array}{l} m\ \text{is an odd element of}\ S\ \text{and}\\ H \neq \mathbb{N} \setminus S\ \text{is an upper $m-$set}\\ \text{with}\ 2 \max(\mathbb{N} \setminus S \cup H) + m \leq F \end{array}\right. \right\}.\]
\end{theorem}

\begin{proof}
If $T \in \mathcal{D}_2(S)$ and $\operatorname{F}(T) \leq F$, then, by Theorem \ref{Th24}, we have that $T = S(m,H)$ for some odd integer $m \in S$ and some upper $m-$set, $H$, of $S$. We distinguish two cases, depending on which $H$ is chosen:
\begin{itemize}
    \item If $H = \mathbb{N} \setminus S$, then, by Proposition \ref{Prop34}, we have that $m > \operatorname{F}(S)$ and, by Proposition \ref{Prop32}, that $m - 2 \leq F$. Therefore, we conclude that $m$ must belong to $\{\operatorname{F}(S)+1, \ldots, F+2\}$.
    \item If $H \neq \mathbb{N} \setminus S$, then, by Proposition \ref{Prop32}, we have that $2 \max(\mathbb{N} \setminus S \cup H) + m \leq F$.
\end{itemize}
This completes the proof.
\end{proof}

Now we can formulate the algorithm to compute $\{T \in \mathcal{D}_2(S) \mid \operatorname{F}(T) \leq F\}$ for given $S \in \mathscr{L}$ and $F \in \mathbb{N} \setminus \{0\}$.

\begin{algorithm}\label{Alg37}
Computation of $\{T \in \mathcal{D}_2(S) \mid \operatorname{F}(T) \leq F\}$.

\textsc{Input:} A numerical semigroup $S$ and a positive integer $F$.

\textsc{Output:} $\{T \in \mathcal{D}_2(S) \mid \operatorname{F}(T) \leq F\}$.

\begin{enumerate}
    \item If $2 \operatorname{F}(S) > F$, then return $\varnothing$.
    \item Set $A = \{m \in \mathbb{N} \mid m\ \text{is odd and}\ \operatorname{F}(S) + 1 \leq m \leq F+2\}$.
    \item Set $B = \{m \in S \mid m\ \text{is odd and}\ m \leq F-2\}$.
    \item For each $m \in B$ define \[H(m)=\left\{ H \left\vert \begin{array}{l} H \neq \mathbb{N} \setminus S\ \text{is an upper $m-$set}\\ \text{such that} \max(\mathbb{N} \setminus S \cup H) \leq \frac{F-m}2 \end{array}\right. \right\}.\]
    \item Return $\{S(m,\mathbb{N} \setminus S) \mid m \in A\} \cup \{S(m,H) \mid m \in B\ \text{and}\ H \in H(m)\}$.
\end{enumerate}   
\end{algorithm}

The following GAP function implements Algorithm \ref{Alg37}. It requires both the package \texttt{NumericalSgps} and the function \texttt{IsUppermSetOfNumericalSemigroup} declared in the previous section.

{\small
\begin{verbatim}
  Algorithm35:=function(S,F)
    local out,FS,gaps,A,B,pow,Hm,Aux,m;
    out:=[];
    FS:=FrobeniusNumber(S);
    if 2*FS > F then return(out); fi;
    gaps:=Gaps(S);
    A:=Filtered([(FS+1)..(F+2)],m->IsOddInt(m));
    out:=List(A,m->[m,gaps]);
    B:=Filtered([1..(F-2)],m->IsOddInt(m));
    B:=Filtered(B,m->BelongsToNumericalSemigroup(m,S));
    pow:=Combinations(gaps);
    pow:=Difference(pow,[gaps]);
    Hm:=[];
    for m in B do
      Aux:=Filtered(pow,H->IsUppermSetOfNumericalSemigroup(S,m,H));
      Hm:=Filtered(Aux,H->Maximum(Difference(gaps,H))<=(F-m)/2);
      Append(out,List(Hm,H->[m,H]));
    od;
    return(out);
  end;
\end{verbatim}
}

\begin{example}
Let $S = \langle 4,5,11 \rangle$ and $F = 15$. Using the GAP function above, we can verify, as follows, that $\{T \in \mathcal{D}_2(S) \mid \operatorname{F}(T) \leq 15\}$ is equal to
\[
\begin{array}{c}
\{S(9, \{1, 2, 3, 6, 7\}), S(11, \{1, 2, 3, 6, 7\}), S(13, \{1, 2, 3, 6, 7\}), S(15, \{1, 2, 3, 6, 7\}),\\ S(17, \{1, 2, 3, 6, 7\}), S(5, \{3, 6, 7\}), S(5, \{6, 7\}), S(9, \{1, 2, 6, 7\}), S(9, \{1, 3, 6, 7\}),\\ S(9, \{1, 6, 7\}), S(9, \{2, 3, 6, 7\}), S(9, \{2, 6, 7\}), S(9, \{3, 6, 7\}), S(9, \{6, 7\}),\\ S(11, \{1, 3, 6, 7\}), S(11, \{2, 3, 6, 7\}), S(11, \{3, 6, 7\}), S(13, \{2, 3, 6, 7\})\}. 
\end{array}
\]
Now, by recalling Notation \ref{Not25} and using the following GAP function
{\small
\begin{verbatim}
  UpperMSetToNumericalSemigroup:=function(S,m,H)
    local msg,T;
    msg:= MinimalGeneratingSystem(S);
    T:=NumericalSemigroup(Union(2*msg,[m],2*H+m));
    return(T);
  end;   
\end{verbatim}
}
\noindent we obtain that the above set is equal to
\[
\begin{array}{c}
\{
\langle 8, 9, 10, 11, 13, 15 \rangle, \langle 8, 10, 11, 13, 15, 17 \rangle, \langle 8, 10, 13, 15, 17, 19, 22 \rangle,\\ \langle 8, 10, 15, 17, 19, 21, 22 \rangle, \langle 8, 10, 17, 19, 21, 22, 23 \rangle, \langle 5, 8, 11, 17 \rangle, \langle 5, 8, 17, 19 \rangle,\\ \langle 8, 9, 10, 11, 13 \rangle, \langle 8, 9, 10, 11, 15 \rangle, \langle 8, 9, 10, 11, 23 \rangle, \langle 8, 9, 10, 13, 15 \rangle,\\ \langle 8, 9, 10, 13 \rangle, \langle 8, 9, 10, 15, 21, 22 \rangle, \langle 8, 9, 10, 21, 22, 23 \rangle, \langle 8, 10, 11, 13, 17 \rangle,\\ \langle 8, 10, 11, 15, 17 \rangle, \langle 8, 10, 11, 17, 23 \rangle, \langle 8, 10, 13, 17, 19, 22 \rangle
\}.
\end{array}
\]
\end{example}

\section{Numerical semigroups with given depth}

\begin{definition}
Let $S$ be a numerical semigroup. The \emph{depth} of $S$, denoted $\operatorname{depth}(S)$, is the integer number $q$ such that $\operatorname{F}(S)+1 = q \operatorname{m}(S) - r$ for some integer $0 \leq r < \operatorname{m}(S).$
\end{definition}

Observe that \[\operatorname{depth}(S) = \left\lceil \frac{\operatorname{F}(S)+1}{\operatorname{m}(S)} \right\rceil = \left\lfloor \frac{\operatorname{F}(S)}{\operatorname{m}(S)} \right\rfloor+1.\] Note that the depth of $S$ matches the so-called complexity of $S$ (see \cite[Corollary 21]{questiones}).

Given $q \in \mathbb{N}$, we write $\mathscr{C}_q$ for the set of numerical semigroups having depth less than or equal to $q$, that is,
\[\mathscr{C}_q = \{S \in \mathscr{L} \mid \operatorname{depth}(S) \leq q\}.\]
Notice that $\mathscr{C}_0 = \{\mathbb{N}\}$ and that $\mathscr{C}_1 = \left\{ \{0,F+1,F+2, \ldots, \} \mid F \in \mathbb{N} \right\}$. Clearly, both $\mathscr{C}_0$ and $\mathscr{C}_1$ are arithmetic varieties. Let us prove that this is true for every $\mathscr{C}_q,\ q \in \mathbb{N}$.

\begin{lemma}\label{lema39}
If $S$ and $T$ are numerical semigroups and $d \in \mathbb{N} \setminus \{0\}$, then \[\operatorname{depth}\left(\frac{S}d\right) \leq \operatorname{depth}(S).\]
\end{lemma}

\begin{proof}
If $d \in S$, then $\frac{S}d = \mathbb{N}$ and, thus, $\operatorname{depth}(\mathbb{N}) = 0 \leq \operatorname{depth}(S).$ If $d \not\in S$, then we have that $\operatorname{m}\left(\frac{S}d\right) \geq \operatorname{m}(S)$ and that $\operatorname{F}\left(\frac{S}d\right) \leq \operatorname{F}(S)$. Therefore, 
\[\operatorname{depth}\left(\frac{S}d\right) = \left\lfloor 
\frac{\operatorname{F}\left(\frac{S}d\right)}{\operatorname{m}\left(\frac{S}d\right)}
\right\rfloor + 1 \leq \left\lfloor
\frac{\frac{\operatorname{F}(S)}{d}}{\frac{\operatorname{m}(S)}{d}}
\right\rfloor + 1 = \operatorname{depth}(S)\]
and we are done.
\end{proof}

\begin{theorem}\label{ThDepth}
The set $\mathscr{C}_q$ is an arithmetic variety for every $q \in \mathbb{N}$.
\end{theorem}

\begin{proof}
Since $\operatorname{depth}(\mathbb{N}) = 0$, we have that $\mathbb{N} \in \mathscr{C}_q$ for every $q \in \mathbb{N}$; in particular, $\mathscr{C}_q \neq \varnothing$ for every $q \in \mathbb{N}$.

Let $q \in \mathbb{N}$. On the one hand, if $\{S,T\} \subseteq \mathscr{C}_q$, then we may suppose that $\operatorname{F}(S) \leq \operatorname{F}(T)$. So,  $\operatorname{F}(S \cap T) = \max(\operatorname{F}(S),\operatorname{F}(T)) = \operatorname{F}(T)$ and, since $\operatorname{m}(S \cap T) \geq \max(\operatorname{m}(S),\operatorname{m}(T))),$ by Lemma \ref{lema39}, we have that 
\[
\operatorname{depth}(S \cap T) = \left\lfloor
\frac{\operatorname{F}(S \cap T)}{\operatorname{m}(S \cap T)} \right\rfloor + 1 = \left\lfloor
\frac{\operatorname{F}(T)}{\operatorname{m}(S \cap T)} \right\rfloor + 1 \leq \left\lfloor \frac{\operatorname{F}(T)}{\operatorname{m}(T)} \right\rfloor + 1 = \operatorname{depth}(T) \leq q
\] and, consequently, that $S \cap T \in \mathscr{C}_q$. On the other hand, if $S \in \mathscr{C}_q$, by Lemma \ref{lema39}, we have that $\operatorname{depth}\left(\frac{S}d\right) \leq \operatorname{depth}(S) \leq q$, that is, $\frac{S}d \in \mathscr{C}_q$. Now, by Proposition \ref{Prop1}, we conclude that $\mathscr{C}_q$ is an arithmetic variety.
\end{proof}

As an immediate consequence of Theorem \ref{Th22} and Proposition \ref{Prop30}, we have that $\mathcal{G}_{{(\mathscr{C}_q)}_F}$ is a finite rooted tree with root $\mathbb{N}$ such that the set of children of $S \in {(\mathscr{C}_q)}_F$ is equal to 
\[\{T \in \mathcal{D}_2(S) \mid \operatorname{F}(T) \leq F\ \text{and}\ \operatorname{depth}(T) \leq q\}.\]
Therefore, we can formulate an algorithm for the computation of ${(\mathscr{C}_q)}_F$.

\begin{algorithm}\label{Alg43}
Computation of ${(\mathscr{C}_q)}_F$.

\textsc{Input:} Two positive integers $F$ and $q$.

\textsc{Output:} The arithmetic variety ${(\mathscr{C}_q)}_F$.

\begin{enumerate}
    \item Set $A = B = \{\mathbb{N}\}$.
    \item While $B \neq \varnothing$ do
    \begin{enumerate}[(2.1)]
        \item For $S \in B$ do
        \begin{enumerate}[(2.1.1)]
            \item Compute $B_S = \{T \in \mathcal{D}_2(S) \mid \operatorname{F}(T) \leq F\} \setminus \{S\}$.
            \item Compute $C_S = \{T \in \mathcal{B}(S) \mid \operatorname{depth}(T) \leq q\}$.
        \end{enumerate}
        \item Set $B = \cup_{S \in B} C_S$.
        \item Set $A = A \cup B$.
    \end{enumerate}
    \item Return $A$.
\end{enumerate}   
\end{algorithm}

\begin{example}
We can easily check that $(\mathscr{C}_2)_5$ is equal to 
\[
\begin{array}{c}
\{\mathbb{N}, \langle 2, 3 \rangle, \langle 2, 5 \rangle, \langle 3,4, 5 \rangle, \langle 3, 4 \rangle, \langle 3, 5, 7 \rangle, \langle 3, 7, 8 \rangle,\\ \langle 4,5,6, 7 \rangle, \langle 4, 6, 7, 9 \rangle, \langle 5,6,7,8, 9 \rangle, \langle 6, 7, 8, 9, 10, 11 \rangle \}
\end{array}
\]
using the following GAP function based in Algorithm \ref{Alg43}. 

{\small
\begin{verbatim}
  NumericalSemigroupsWithFrobeniusNumberAndDepth:=function(F,q)
    local A,B,C,depth,S,BS,rat,CS;
    A:=[NumericalSemigroup(1)];
    B:=A;
    while not(IsEmpty(B)) do
      C:=[];
      for S in B do 
        BS := Algorithm35(S,F);
        BS := Set(BS,i->UpperMSetToNumericalSemigroup(S,i[1],i[2]));
        BS := Difference(BS,[NumericalSemigroup(1)]);
        rat:=function(i)
          return (FrobeniusNumber(i)+1)/MultiplicityOfNumericalSemigroup(i);
        end;
        CS := Filtered(BS,i->CeilingOfRational(rat(i)) <= q);
        C:=Union(C,CS);
      od;
      B:=C;
      A:=Union(A,B);
    od;
    return(A);
  end;  
\end{verbatim}
}
\end{example}

\subsection*{Data availability}\mbox{}\par
\noindent Data sharing is not applicable to this article, as no datasets were generated or analyzed during the present work.

\subsection*{Declarations}\mbox{}\par 
\noindent The authors are partially supported by Proyecto de Excelencia de la Junta de Andaluc\'{\i}a (ProyExcel\_00868) and by Proyecto de investigaci\'on del Plan Propio - UCA 2022-2023 (PR2022-011). The second author is partially supported by grant PID2022-138906NB-C21 funded by MICIU/AEI/ 10.13039/501100011033 and by ERDF/EU, and by research group FQM024 funded by Junta de Extremadura (Spain)/FEDER funds.

\medskip
\noindent The authors have no relevant financial or non-financial interests to disclose.

\medskip
\noindent
All authors have contributed equally in the development of this work.

\end{document}